\documentclass[12pt,leqno]{amsart}
\usepackage{amsfonts,amsthm,amsmath,xcolor,enumitem,comment}
\theoremstyle{plain}
\newtheorem{thm}{Theorem}[section]
\newtheorem{prop}{Proposition}[section]
\newtheorem{lem}{Lemma}[section]

\theoremstyle{definition}
\newtheorem{df}{Definition}[section]
\newtheorem{rem}{Remark}[section]
\newtheorem{ex}{Example}[section]

\newcommand{\FF}{\mathbb{F}}

\newcommand{\RR}{\mathbb{R}}

\newcommand{\D}{\mathcal{D}}
\newcommand{\R}{\mathbb{R}}

\newcommand{\la}{\langle}
\newcommand{\ra}{\rangle}

\newcommand{\bbinom}[2]{\left[\!\begin{array}{c} #1 \\ #2 \end{array}\!\right]}

\def\bm#1{\mathbf{#1}}

\DeclareMathOperator{\supp}{supp}
\DeclareMathOperator{\Supp}{Supp}

\DeclareMathOperator{\wt}{wt}

\DeclareMathOperator{\Harm}{Harm}

\DeclareMathOperator{\Mat}{Mat}

\begin{document}

\title[Higher and extended Jacobi polynomials]{Higher and extended Jacobi polynomials for codes}

\author[Chakraborty]{Himadri Shekhar Chakraborty*}
\thanks{*Corresponding author}
\address
{
	Department of Mathematics, 
	Shahjalal University of Science and Technology\\ Sylhet-3114, Bangladesh\\
}
\email{himadri-mat@sust.edu}

\author[Miezaki]{Tsuyoshi Miezaki}
\address
{
	Faculty of Science and Engineering, 
	Waseda University, 
	Tokyo 169-8555, Japan\\
}
\email{miezaki@waseda.jp}

\date{}
\maketitle

\begin{abstract}
In this paper, we introduce Jacobi polynomial generalizations of 
several classical invariants in coding theory over finite fields, 
specifically, the higher and extended weight enumerators,
and we establish explicit correspondences 
between the resulting Jacobi polynomials.
Moreover, we present the Jacobi analogue of MacWilliams identity 
for both higher and extended weight enumerators. 
We also present that the higher Jacobi polynomials for linear codes 
whose subcode supports form $t$-designs 
can be uniquely determined
from the higher weight enumerators of the codes 
via polarization technique.
Finally, 
we demonstrate how higher Jacobi polynomials
can be computed 
from harmonic higher weight enumerators
with the help of Hahn polynomials.
\end{abstract}

{\small
\vspace{-0.15in}

\quad
}

%\noindent
%2010 {\it Mathematics Subject Classification}.
%Primary 11F30;
%Secondary 20D08, 11F27.\\ \quad

\section{Introduction}

The concept of higher weights of a linear code over a finite field 
is a remarkable generalization of the classical Hamming weights.
Wei~\cite{Wei1991} introduced the notion of higher weights 
in the study of linear codes over finite fields. 
As the theory developed,
it led to natural generalizations of 
classical results in coding
theory. In particular,
Kl{\o}ve~\cite{Klove1992} and
Simonis~\cite{Simonis1993} 
independently gave a celebrated generalization of 
the MacWilliams identity (cf.~\cite{MacWilliams})
for these higher weights.
Motivated by Kl{\o}ve's proof of higher weight generalization of 
the MacWilliams identity 
that involves the use of extension fields,
Jurrius and Pellikaan~\cite{JP2013}
defined extension codes
which extend the theory of linear codes 
to those defined over finite field extensions.
Moreover, Britz and Shiromoto~\cite{BrSh2008} 
presented a construction of $t$-designs 
with respect to 
the subcode supports of linear codes over finite fields,
and
generalized the Assmus-Mattson Theorem (cf.~\cite{AsMa69})
for such subcodes; 
see also \cite{Tanabe2001}.

The coding theoretical analogue of Jacobi forms (cf.~\cite{EZ}) are 
known as Jacobi polynomials, representing a significant
generalization of the weight enumerator of a code.
%Note that one of most remarkable generalization of 
%Siegel modular forms (cf.~\cite{K}) in lattice theory (cf.~\cite{CS}) 
%are Jacobi forms. 
The notion of Jacobi polynomials for linear codes 
with respect to a reference vector
was introduced by Ozeki~\cite{Ozeki}.
He also presented the Jacobi analogue 
the MacWilliams identity for linear codes. 
Surprisingly,
Bannai and Ozeki~\cite{BO} applied polarization techniques 
to study Jacobi polynomials
for linear codes whose codewords hold $t$-designs. 
%Recently, the concept of Jacobi polynomials 
%attached to multiple reference vectors 
%was defined in~\cite{CMOT2023}. 
Later, Bonnecaze, Mourrain and Sol\'e~\cite{BoMoSo1999} 
redefined the Jacobi polynomials 
in the sense of coordinate positions.
They pointed out that
the Jacobi polynomials for linear codes whose codewords hold $t$-designs 
can be uniquely determined
from the weight enumerators of the codes with the help of polarization operator.
Many authors studied several generalizations of the Jacobi polynomials for codes;
for instance, see~\cite{CIT2024, CM2021, CMO2022, CMOT2023, HOO2020}.
%Among them the higher genus generalizations were
%studied in~\cite{CM2021, CMO2022, HOO2020}
%and split type generalizations were 
%studied in~\cite{CIT2024}. 
%For more discussions on Jacobi polynomials,
%we refer the reader to~\cite{CM2021, CMOT2023}.

%Delsarte~\cite{Delsarte} introduced discrete harmonic functions on a finite set.
Bachoc~\cite{Bachoc} made use of 
Delsarte's~\cite{Delsarte} 
concept of discrete harmonic functions %on a finite set
to define harmonic weight enumerators of binary codes, and presented 
a striking generalization of the MacWilliams identity 
for it. 
%associated to a discrete harmonic function. 
She also presented a computation of Jacobi polynomials
using harmonic weight enumerators and Hahn polynomials (cf.~\cite{Hahn1949}).
Recently, Britz, Chakraborty and Miezaki~\cite{BrChMi2026}
introduced harmonic higher and extended weight enumerators of 
linear codes over finite fields
and presented the harmonic generalization of the MacWilliams identity
for higher weights. 
Moreover, they gave a new proof of the Assmus-Mattson Theorem
for subcode support designs
using harmonic higher weight enumerators.

In this paper, we present the Jacobi generalizations of 
well-known polynomials of codes over finite fields, 
namely the higher weight enumerators and the extended weight enumerators,
and we derive the correspondences between these Jacobi polynomials.
Moreover, we present the Jacobi analogue of MacWilliams identity 
for the higher (resp. extended) weight enumerators. 
We also present that the higher Jacobi polynomials for linear codes 
whose subcode supports hold $t$-designs 
can be uniquely determined
from the higher weight enumerators of the codes 
through polarization technique.
%As an application of the subcode designs, 
%we derive a formula that obtains 
%the higher Jacobi polynomials of a linear code over finite fields
%using polarization.
Finally, 
we compute 
higher Jacobi polynomials from harmonic higher weight enumerators
with the help of Hahn polynomials.

This paper is organized as follows. 
In Section~\ref{Sec:Preli}, 
we present basic definitions and properties 
of linear codes that are frequently used in this paper.
In Section~\ref{Sec:GenJac}, 
we define the higher (resp.\ extended) Jacobi polynomials 
for codes over finite fields,
and give formulae to compute these; 
see Theorems~\ref{Thm:HigherJacReinter} and~\ref{Thm:ExtendedReinter}.
Moreover, we present the MacWilliams-type identity for 
higher (resp.\ extended Jacobi polynomials of codes 
over finite fields; 
see Theorems~\ref{Thm:HarmHigherMac} and~\ref{Thm:JacExtMacWilliams.}. 
In Section~\ref{Sec:Connection},
we provide the relationships between these polynomials;
see Theorems~\ref{Thm:ExtendedtoHigher} and~\ref{Thm:HigherToExtended}.  
In Section~\ref{Sec:Design}, 
we observe how polarization technique can be
applied to obtain the higher Jacobi polynomials 
attached to coordinate places of a code
over finite fields;
Theorems~\ref{Thm:JacToDesign}, \ref{Thm:HigherJac_1_Design} and~\ref{Thm:HigherJac_t_design}.
Finally, in Section~\ref{Sec:CompJac},
we compute 
higher Jacobi polynomials from harmonic higher weight enumerators
with the help of Hahn polynomials.

\section{Preliminaries}\label{Sec:Preli}

In this section, we present basic definitions and notation for linear codes 
that are frequently needed in this paper. 
We mostly follow the {definitions and notation} 
of~\cite{BrSh2008, HP2003, MS1977}.

Let $\FF_{q}$ be a finite field of order~$q$, where~$q$ is a 
prime power.  
Then $\FF_{q}^{n}$ denotes the vector space of 
dimension~$n$ with the usual inner product:
$\bm{u}\cdot\bm{v} := u_{1}v_{1} + \cdots + u_{n}v_{n}$
for $\bm{u},\bm{v} \in \FF_{q}^{n}$,
where
$\bm{u} = (u_{1},\ldots,u_{n})$ and $\bm{v} = (v_{1},\ldots,v_{n})$.
For $\bm{u} \in \FF_{q}^{n}$, we call
$\supp(\bm{u}) := \{i \in E \mid u_{i} \neq 0\}$
the \emph{support} of $\bm{u}$,
and
$\wt(\bm{u}) := |\supp(\bm{u})|$
the \emph{weight} of $\bm{u}$. 
Similarly, the \emph{support} and \emph{weight}
of each subset of vectors $D \subseteq \FF_{q}^{n}$ are defined as follows:
\begin{align*}
	\Supp(D) 
	& :=
	\bigcup_{\bm{u} \in D}
	\supp(\bm{u})\\
	\wt(D) 
	& :=
	|\Supp(D)|.
\end{align*}

An $\FF_{q}$-\emph{linear code} of length~$n$ is a linear subspace of $\FF_{q}^{n}$. 
Moreover, an $[n,k,d_{\min}]$~code 
is an $\FF_q$-linear code of length~$n$ with dimension~$k$ 
and minimum weight~$d_{\min}$.
Frequently, we call $[n,k,d_{\min}]$~codes $[n,k]$~codes.
The \emph{dual code}~$C^{\perp}$ of an $\FF_{q}$-linear code~$C$  
can be defined as follows:
\[
	C^{\perp}
	:=
	\{
	\bm{u} \in \FF_{q}^{n}
	\mid
	\bm{u} \cdot \bm{v} = 0
	\mbox{ for all }
	\bm{v} \in C
	\}.
\] 

Let $C$ be an $[n,k]$ code. 
Let $r,i$ be positive integers such that $r \leq k$ and $i \leq n$.
Now we define
\begin{align*}
	\D_{r}(C)
	& :=
	\{
	D \mid D \text{ is an } [n,r] \text{ subcode of } C 
	\},\\
	\mathcal{S}_{r}(C)
	& :=
	\{
	\Supp(D) \mid D \in \D_{r}(C)
	\},\\
	\mathcal{S}_{r,i}(C)
	& :=
	\{
	X \in \mathcal{S}_{r}(C) \mid |X| = i
	\}.
\end{align*}
Note that, in general, 
the sets $\mathcal{S}_{r}(C)$ and $\mathcal{S}_{r,i}(C)$ are multisets.
Also, the support of a nonzero codeword $\bm{v}$ 
is identical to that of its span~$\la \bm{v} \ra$. 
Conversely, each member of $\mathcal{S}_{1}(C)$ is the
support of some nonzero codeword. 
Therefore, the sets $\mathcal{S}_{1}(C),\ldots,\mathcal{S}_{k}(C)$ 
extend the notion of codeword support.

Let $C$ be an $\FF_{q}$-linear code of length~$n$.
Then the $r$-th generalized Hamming weight of~$C$
for any~$r$ such that $r \leq k$, is defined as
\[
	d_{r}
	:=
	d_{r}(C)
	:=
	\min
	\{
		\wt(D)
		\mid
		D \in \mathcal{D}_{r}(C)
	\}.
\]
Moreover, the \emph{$r$-th higher weight distribution} of $C$ is the sequence 
$$\{A_{i}^{(r)}(C)\mid i=0,1, \dots, n \},$$ 
where $A_{i}^{(r)}(C)$ is the number of subcodes with a given weight~$i$
and dimension~$r$. That is,
\[
	A_{i}^{(r)}(C)
	:=
	\#
	\{
		D\in \mathcal{D}_{r}(C)
		\mid
		\wt(D) = i
	\}.
\]
Then the polynomial
\[
	W_{C}^{(r)}(x,y) 
	:= 
	\sum^{n}_{i=0} 
	A_{i}^{(r)}(C) x^{n-i} y^{i}
\]
is called the \emph{$r$-th higher weight enumerator} of $C$
and satisfies the following MacWilliams-type identity
stated in~\cite[{Theorem 5.14}]{JP2013}:
\[
	W_{C^{\perp}}^{(r)}(x,y)
	=
	\sum_{j=0}^{r}
	\sum_{\ell=0}^{j}
	(-1)^{r-j}
	\frac{q^{{r-j \choose 2}-j(r-j)-\ell(j-\ell)-jk}}{[r-j]_{q} [j-\ell]_{q}}
	W_{C}^{(\ell)}(x+(q^{j}-1)y,x-y).
\]
\noindent
Here, for all integers $a,b\geq 0$, 
\[
	[a]_q 
	:= 
	[a,a]_q
	\qquad\textrm{where}\qquad
	[a,b]_{q} 
	:=
	\prod_{i=0}^{b-1}
	\big(q^{a}-q^{i}\big).
\]

%Let $C$ be an $\FF_{q}$-linear code of length~$n$.
%Then for an arbitrary subset $X \subseteq E$ and integer~$r \geq 0$,
%we define
%\begin{align*}
%	C(X) 
%	& := 
%	\{
%	\bm{u} \in C 
%	\mid 
%	u_{i} = 0 \text{ for all } i \in X
%	\}\,,\\
%	\ell(X) 
%	& := 
%	\dim C(X)\,,\\
%	B_{X}^{(r)}(C)
%	& := 
	%	q^{\ell(J)}-1
%	\#
%	\{
%	D \subseteq C(X)
%	\mid 
%	D \text{ is a subspace of dimension } r
%	\}\,. 
%\end{align*}
\noindent
Also, define 
\begin{align*}
%	[a,b]_{q} 
%	&:=
%	\prod_{i=0}^{b-1}
%	q^{a}-q^{i},\\
%	[a]_{q}
%	&:=
%	[a,a]_{q},\\
	{\bbinom{a}{b}}_{q}
	&:=
	\frac{[a,b]_{q}}{[b]_{q}}
      =
    \prod_{i=0}^{b-1}\frac{ q^{a}-q^{i}}{q^{b}-q^{i}}\,.	
\end{align*}

The following two {lemmas} provide important and useful identities.
{The second lemma follows immediately from~\cite[Theorem 25.2]{vLiWi01}.}

\begin{lem}\label{Rem:TBritz}
	
	$[a,b]_{q} 
	= 
	\displaystyle
    \sum_{i=0}^{b} 
	{\bbinom{b}{i}}_{q}
	(-1)^{b-i}
	q^{\binom{b-i}{2}}
	(q^a)^i$.
\end{lem}

\section{Generalizations of Jacobi polynomials}\label{Sec:GenJac}

Ozeki~\cite{Ozeki} first defined the notion of 
Jacobi polynomials for codes over finite field 
and presented a MacWilliams-type identity for these polynomials. 
Later, several articles extended this concept 
for higher genus cases~\cite{CHMO2025, CM2021} 
as well as for multiple reference sets~\cite{CIT2024, CMOT2023}.
In this section, 
we introduce the concept of higher (resp. extended) 
Jacobi polynomials for codes over finite fields.

Let $[n] := \{1,2,\ldots,n\}$ be a finite set.
We denote by~$2^{[n]}$ the set of all subsets of~$[n]$,
while for all $t = 0, 1, \ldots, n$,
$\binom{[n]}{t}$ is 
the set of all $t$-subsets of $[n]$.
Let $T \subseteq [n]$.
For $\bm{u} \in \FF_{q}^{n}$,
$\supp_{T}(\bm{u})$ and $\wt_{T}(\bm{u})$
denote the support and weight of~$\bm{u}$ on~$T$, respectively.
Let $C$ be an $[n,k]$ code over~$\FF_{q}$.
For any $D \in \mathcal{D}_{r}(C)$, where $0 \leq r \leq k$,
$\Supp_{T}(D)$ and $\wt_{T}(D)$ denote the support
and weight of~$D$ on~$T$, respectively.
If $T = [n]$, then $\supp_{T}(\bm{u})$ ($\Supp_{T}(D)$) 
and $\wt_{T}(\bm{u})$ (resp. $\wt_{T}(D)$)
coincide with $\supp(\bm{u})$ (resp. $\Supp(D)$) 
and $\wt(\bm{u})$ ($\wt(D)$), respectively.
Moreover, we denote $\overline{T} := [n]\backslash T$.

\begin{df}\label{Def:JacOne}
	Let $C$ be an $\FF_{q}$-linear code of length~$n$. Then 
	the \emph{Jacobi polynomial} attached to a set $T$ of 
	coordinate places of the code~$C$ is defined as follows:
	\[
		J_{C,T}(w,z,x,y) 
		:=
		\sum_{\bm{u}\in C}
		w^{m_0(\bm{u})}
		z^{m_1(\bm{u})}
		x^{n_0(\bm{u})}
		y^{n_1(\bm{u})},
	\]
	where $T\subseteq [n]$, and for $\bm{u}\in C$,
	\begin{align*}
		m_0(\bm{u}) & := \#\{i\in T \mid u_i=0\},\\
		m_{1}(\bm{u}) & := \#\{i \in T \mid u_{i} \neq 0\},\\
		n_0(\bm{u}) & := \#\{i\in \overline{T} \mid u_i= 0\},\\
		n_{1}(\bm{u})  & := \#\{i\in \overline{T} \mid u_i \neq 0\}.
	\end{align*}
Equivalently, we can express $J_{C,T}(w,z,x,y)$ as follows:
\begin{align*}
	J_{C,T}(w,z,x,y)
	& =
	\sum_{\ell = 0}^{n}\,
	\sum_{j = 0}^{|T|}\,
	n_{\ell,j}(C,T)\,
	w^{|T|-j}\,
	z^{j}\,
	x^{n-|T|-\ell+j}\,
	y^{\ell-j},\\
	& =
	\sum_{i = 0}^{n-|T|}\,
	\sum_{j = 0}^{|T|}\,
	A_{i,j}(C,T)\,
	w^{|T|-j}\,
	z^{j}\,
	x^{n-|T|-i}\,
	y^{i},
\end{align*}
where %$B_{i,j}(C,T)$ is given by
\begin{align*}
	n_{\ell,j}(C,T)
	& :=
	\#
	\{
		\bm{u} \in C
		\mid
		\wt(\bm{u}) = \ell
		\text{ and }
		\wt_{T}(\bm{u}) = j		
	\},\\
	A_{i,j}(C,T)
	& :=
	\#
	\{
	\bm{u} \in C
	\mid
	\wt_{\overline{T}}(\bm{u}) = i
	\text{ and }
	\wt_{T}(\bm{u}) = j		
	\}.
\end{align*}
\end{df}

\begin{rem}
	\(n_{\ell, j}(C,T) = A_{\ell-j,j}(C,T)\).
\end{rem}

\begin{rem}
	If $T \subseteq [n]$ is empty, then $J_{C,T}(w,z,x,y)=W_C(x,y)$.
\end{rem}

\subsection{Higher Jacobi polynomials}

$ $\\[1mm]
Let $C$ be an $[n,k]$ code over $\FF_{q}$ and let $T \subseteq [n]$.

\begin{df}\label{Def:HigherJacobi}
	The \emph{$r$-th higher Jacobi polynomial} of $C$ associated to $T$ is
	defined as follows:	
	\[
		J_{C,T}^{(r)}(w,z,x,y)
		=
		\sum_{i = 0}^{n-|T|}\,
		\sum_{j = 0}^{|T|}\,
		A_{i,j}^{(r)}(C,T)\,
		w^{|T|-j}\,
		z^{j}\,
		x^{n-|T|-i}\,
		y^{i},
	\]
	where %$B_{i,j}(C,T)$ is given by
	\(
		A_{i,j}^{(r)}(C,T)
		:=
		\#
		\{
			D \in \D_{r}(C)
			\mid
			\wt_{\overline{T}}(D) = i
			\text{ and }
			\wt_{T}(D) = j		
		\}.
	\)
\end{df}

\begin{rem}
	Clearly,
	$A_{0,0}^{(0)}(C,T) = 1$ and $A_{0,0}^{(r)}(C,T) = 0$ for 
	$r \neq 0$.
\end{rem}

\begin{rem}
%	Let $C$ be an $[n,k]$ code and let $f \in \Harm_{d}(n)$, where $d \neq 0$.
	Since every $1$-dimensional subspace of $C$ contains $q-1$ nonzero codewords, 
    $(q - 1)A_{i,j}^{(1)}(C,T) = A_{i,j}(C,T)$ for $0 \leq i \leq n-|T|$ and 
    $0 \leq j \leq |T|$. 
	Therefore,
	\[
		J_{C,T}(w,z,x,y) 
		= 
		J_{C,T}^{(0)}(w,z,x,y)
		+
		J_{C,T}^{(1)}(w,z,x,y).
	\]
\end{rem}

\begin{thm}[MacWilliams type identity]\label{Thm:HarmHigherMac} 
	{Let $C$ be an $[n,k]$ code over $\FF_{q}$,
	and let $T \subseteq [n]$.
	Then }
	\begin{multline*}
		J_{C^{\perp},T}^{(r)}
		(w,z,x,y)
		=  
		\sum_{j=0}^{r}
		\sum_{\ell=0}^{j}
		(-1)^{r-j}
		\frac{q^{{r-j \choose 2}-j(r-j)-\ell(j-\ell)-jk}}{[r-j]_{q} [j-\ell]_{q}}\\
		J_{C,T}^{(\ell)}
		\left( 
			{w+(q^{j}-1)z},\, 
			{w-z},\,
			{x+(q^{j}-1)y},\, 
			{x-y}
		\right).
	\end{multline*}
\end{thm}

%In order to give a proof of the above theorem, %\ref{Thm:HarmTutteMacIden.} 
{To pave the way for the proof of Theorem~\ref{Thm:HarmHigherMac}
that we give 
%in Section~\ref{Sec:Greene},
in Section~\ref{Sec:Connection},
we begin by presenting some propositions and theorems
%harmonic functions
%and their associated polynomials 
related to~$J_{C,T}^{(r)}(w,z,x,y)$.}

Let $C$ be an $\FF_{q}$-linear code of length~$n$.
Let $T \subseteq [n]$.
Then for arbitrary subsets 
$X \subseteq \overline{T}$ and $Y \subseteq T$, and integer~$r \geq 0$,
we define
\begin{align*}
	C_{T}(X,Y) 
	& := 
	\{
		\bm{u} \in C 
		\mid 
		u_{i} = 0 \text{ for all } i \in X \cup Y
	\}\,,\\
	\ell_{T}(X,Y) 
	& := 
	\dim C_{T}(X,Y)\,,\\
	B_{X,Y}^{(r)}(C,T)
	& := 
	\#
	\{
	D \in \D_{r}(C)
	\mid
	D \subseteq C_{T}(X,Y)
	\}\,. 
\end{align*}

\begin{lem}\label{Rem:BXrC}
	$B_{X,Y}^{(r)}(C;T) = {\bbinom{\ell_{T}(X,Y)}{r}}_{q}$.
\end{lem}

\begin{prop}\label{Prop:Connection}
	Let $C$ be an~$[n,k]$ code over~$\FF_{q}$ and $T \subseteq [n]$.
	Define
	\[
		Q_{s,t}^{(r)}(C;T)
		:= 
		\sum_{X \in \binom{\overline{T}}{s}}\,
		\sum_{Y \in \binom{T}{t}} 
		B_{X,Y}^{(r)}(C;T).
	\]
	{Then
	we have the following relation between 
	$Q_{s,t}^{(r)}(C,T)$ and $A_{i,j}^{(r)}(C,T)$:}
	\[
		Q_{s,t}^{(r)}(C;T) 
		= 
		\sum_{i=0}^{n-|T|-s}\,
		\sum_{j=0}^{|T|-t} \,
		\binom{n-|T|-i}{s}\,
		\binom{|T|-j}{t}\,
		A_{i,j}^{(r)}(C;T).
	\]
\end{prop}

\begin{proof}
	For $T \subseteq [n]$, we will count the following set:
	\begin{multline*}
		\mathcal{B}_{s,t}^{(r)}(C;T)\\
		=
		\{
			(D;X,Y)
			\mid
			X \in \binom{\overline{T}}{s},\,
			Y \in \binom{T}{t},
			D \in \D_{r}(C),
			D \subseteq C_{T}(X,Y))
		\}
	\end{multline*}
	in two different ways.
	For each pair of $(X,Y)$ with $|X| = s$ and $|Y| = t$
	there are $B_{X,Y}^{(r)}(C;T)$ 
	tuple~$(D;X,Y)$ in $\mathcal{B}_{s,t}^{(r)}(C;T)$. 
	So, the total number of elements in the set is:
	\[
		Q_{s,t}^{(r)}(C;T)
		= 
		\sum_{X \in \binom{\overline{T}}{s}}\,
		\sum_{Y \in \binom{T}{t}} 
		B_{X,Y}^{(r)}(C;T).
	\]
	On the other hand, let $D \in \D_{r}(C)$ with
	$\wt_{[n]\backslash T}(D) = i$
	and~$\wt_{T}(D) = j$.
	We can find such~$D$ in $A_{i,j}^{(r)}(C;T)$ ways.
	So, the number of tuple~$(D;X,Y)$ in $\mathcal{B}_{s,t}^{(r)}(C;T)$
	such that $D \subseteq C_{T}(X,Y)$,
	is as follows:
	\[
		\binom{n-|T|-i}{s}\,
		\binom{|T|-j}{t}\,
		A_{i,j}^{(r)}(C;T).
	\]
	Summing the above expression 
	over all $i$ ($0 \leq i \leq n-|T|$) 
	and $j$ ($0 \leq j \leq |T|$)
	proves the given identity.	
\end{proof}

\begin{rem}
	For~$r = 0$, we have from the above proposition that 
	\[
		Q_{s,t}^{(0)}(C;T) 
		= 
		\binom{n-|T|}{s} 
		\binom{|T|}{t}.
	\]
	So, in general, we have from Lemma~\ref{Rem:BXrC} that 
	$\ell_{T}(X,Y) = 0$ does that implies 
	$B_{X,Y}^{(r)}(C;T) = 0$.
	However, if $r \neq 0$ and $\ell_{T}(X,Y) = 0$
	implies $B_{X,Y}^{(r)}(C;T) = 0$ and $Q_{s,t}^{(r)}(C;T) = 0$.
\end{rem}

\noindent
Now we have the following result.
\begin{thm}\label{Thm:HigherJacReinter}
	Let $C$ be an $[n,k]$ code $C$ over $\FF_{q}$,
    and let $T \subseteq [n]$.
	Then
	\[
		J_{C,T}^{(r)}(w,z,x,y)
		= 
		\sum_{s = 0}^{n-|T|}\,
		\sum_{t = 0}^{|T|}\,
		Q_{s,t}^{(r)}(C;T)\,
		(w-z)^{t}\, z^{|T|-t}
		(x-y)^{s}\, y^{n-|T|-s}\,.
	\]
\end{thm}

\begin{proof}
	By Proposition~\ref{Prop:Connection}
    and the binomial expansion of 
	$x^{n-i} = ((x-y)+y)^{n-i}$,
	we have
	\begin{align*}
		&
		\sum_{s = 0}^{n-|T|}\,
		\sum_{t = 0}^{|T|}\,
		Q_{s,t}^{(r)}(C;T)\,
		(w-z)^{t}\, z^{|T|-t}
		(x-y)^{s}\, y^{n-|T|-s}\,\\ 
		= & 
		\sum_{s = 0}^{n-|T|}\,
		\sum_{t = 0}^{|T|}\,
		\sum_{i=0}^{n-|T|-s}\,
		\sum_{j=0}^{|T|-t} \,
		\binom{n-|T|-i}{s}\,
		\binom{|T|-j}{t}\,\\\
		&\hspace{50pt}
		A_{i,j}^{(r)}(C;T)\,
		(w-z)^{t}\, z^{|T|-t}
		(x-y)^{s}\, y^{n-|T|-s}\,\\
		= & 
		\sum_{i=0}^{n-|T|}\,
		\sum_{j=0}^{|T|}\,
		A_{i,j}^{(r)}(C;T)\,
		\sum_{s = 0}^{n-|T|-i}\,
		\sum_{t = 0}^{|T|-j}\, 
		\binom{n-|T|-i}{s}\,
		\binom{|T|-j}{t}\,\\
		&\hspace{50pt}
		(w-z)^{t}\, z^{|T|-j-t}\, 
		(x-y)^{s}\, y^{n-|T|-i-s}\,
		z^{j}\,y^{i}\\
		= & 
		\sum_{i=0}^{n-|T|}\,
		\sum_{j=0}^{|T|}\,
		A_{i,j}^{(r)}(C;T)\,
		\left(
			\sum_{s = 0}^{n-|T|-i}\,
			\binom{n-|T|-i}{s}\,
			(x-y)^{s}\, y^{n-|T|-i-s}\,
		\right)\\
		&\hspace{50pt}
		\left(
			\sum_{t = 0}^{|T|-j}\,
			\binom{|T|-j}{t}\,
			(w-z)^{t}\, z^{|T|-j-t}\,
		\right)
		z^{j}\,y^{i}\\
		= &  
		\sum_{i=0}^{n-|T|}\,
		\sum_{j=0}^{|T|}\,
		A_{i,j}^{(r)}(C;T)\,
		w^{|T|-j}\,
		z^{j}\,
		x^{n-|T|-i}\,
		y^{i}\\
		= & \,
		J_{C,T}^{(r)}(w,z,x,y).\qedhere
	\end{align*}
%	since $A_{i,f}^{(r)}(C) = 0$ for $i<d$ and $i > n-d$.
\end{proof}

\subsection{Extended Jacobi polynomials}\label{Sec:ExtenJac}

Let~$m$ be a positive integer. 
Then any~$[n,k]$ code~$C$ over $\FF_{q}$
can be extended to a linear code $C{[q^m]}:=C\otimes \FF_{q^{m}}$
over~$\FF_{q^{m}}$ by considering all the $\FF_{q^{m}}$-linear combinations
of the codewords in~$C$. 
We call~$C{[q^m]}$ the $\FF_{q^{m}}$-\emph{extension code} of~$C$. 
%over~$\FF_{q^{m}}$.
Clearly, $\dim_{\FF_{q}} C(X) = \dim_{\FF_{q^{m}}} C[q^m](X)$ 
for all $X \subseteq [n]$ (see~\cite[{p.~34}]{JP2013}).
%Now we have the following result.

\begin{df}
	Let $C$ be an $[n,k]$ code $C$ over $\FF_{q}$,
    and let $T \subseteq [n]$.
	Then the \emph{$q^m$-extended Jacobi polynomial}
	of~$C$ associated to~$T$ can be defined as follows:
	\[
		J_{C,T}(w,z,x,y;q^m)
		:=
		\sum_{i=0}^{n-|T|}\,
		\sum_{j=0}^{|T|}\,
		A_{i,j}^{q^m}(C;T)\,
		w^{|T|-j}\, z^{j}\,
		x^{n-|T|-i}\, y^{i},
	\]
	where 
	\[
		A_{i,j}^{q^m}(C;T)
		:=
		\#
		\{
			\bm{u} \in C[q^m]
			\mid
			\wt_{\overline{T}}(\bm{u}) = i
			\text{ and }
			\wt_{T}(\bm{u}) = j		
		\}.
	\]
\end{df}

%{Now we have the following Jacobi analogue of 
%the MacWilliams identity 
%for $\FF_{q^{m}}$-extended
%code. Ozeki~\cite{Ozeki} first
%presented the Jacobi generalizations of 
%the MacWilliams identity 
%for codes over~$\FF_{q}$.
%Later, several articles gave various generalizations
%of this Jacobi analogue of MacWilliams identity,
%in particular, for higher genus cases~\cite{bibid} 
%and multiple reference sets~\cite{bibid}. 

The following theorem gives the Jacobi analogue
of MacWilliams identity for $\FF_{q^{m}}$-extended weight enumerators
of codes over~$\FF_{q}$.
Since the proof of the theorem
is straightforward, we omit it.

\begin{thm}[MacWilliams type identity]\label{Thm:JacExtMacWilliams.} 
	Let $C$ be an $[n,k]$ code over $\FF_{q}$,
	and let $T \subseteq [n]$.
	Then
	\begin{multline*}
		J_{C^{\perp},T}
		(w,z,x,y;q^m)\\
		= 
		q^{-km} 
		J_{C,T}
		\left( 
			{w+(q^{m}-1)z},\, 
			{w-z},\,
			{x+(q^{m}-1)y},\, 
			{x-y};\,
			q^m
		\right).
	\end{multline*}
\end{thm}

\begin{df}
	Let $C$ be an $\FF_{q}$-linear code of length~$n$ and
	let $T \subseteq [n]$. 
	For any subsets $X \subseteq \overline{T}$ 
	and $Y \subseteq T$ and positive integer~$m$,
	we define
	\begin{align*}
		B_{X,Y}^{q^m}(C;T)
		& := 
		(q^m)^{\ell_{T}(X,Y)},\\
		Q_{s,t}^{q^m}(C;T)
		& := 
		\sum_{X \in \binom{\overline{T}}{s}}\,
		\sum_{Y \in \binom{T}{t}} 
		B_{X,Y}^{q^m}(C;T).
	\end{align*} 
	Note that $B_{X,Y}^{q^m}(C;T)$ denotes the number of codewords in
	$(C[\FF_{q}])_{T}(X,Y)$.
\end{df}

\begin{prop}\label{Prop:JacTQAConnection}
	The relation between $Q_{s,t}^{q^m}(C;T)$ and $A_{i,j}^{q^m}(C;T)$ 
	is as follows:
	\[
		Q_{s,t}^{q^m}(C;T) 
		= 
		\sum_{i=0}^{n-|T|-s}\,
		\sum_{j=0}^{|T|-t} \,
		\binom{n-|T|-i}{s}\,
		\binom{|T|-j}{t}\,
		A_{i,j}^{q^m}(C;T).
	\]
\end{prop}
	
	\begin{proof}
		For $T \subseteq [n]$, 
		by counting the number of elements the following set:
		\[
			\{
				(\bm{u};X,Y)
				\mid
				X \in \binom{\overline{T}}{s},\,
				Y \in \binom{T}{t},
				\bm{u} \in (C[\FF_{q^{m}}])_{T}(X,Y)
			\}
		\]
		in two different ways, 
		we can prove the identity.
\end{proof}
Using the above proposition, 
%we get the following relation
we have the extended analogue of Theorem~\ref{Thm:HigherJacReinter} as follows.
Since the proof is straightforward, we leave it to the reader.

\begin{thm}\label{Thm:ExtendedReinter}
	Let $C$ be an $\FF_{q}$-linear code of length~$n$, 
	and let $T \subseteq [n]$.
	Then 
	\begin{multline*}
		J_{C,T}(w,z,x,y;q^m)\\
		= 
		\sum_{s=0}^{n-|T|}\,
		\sum_{t=0}^{|T|}\,
		Q_{s,t}^{q^m}(C;T)\,
		(w-z)^{t}\, z^{|T|-t}
		(x-y)^{s}\, y^{n-|T|-s}.
	\end{multline*}
\end{thm}

%\begin{proof}
%	By using Proposition~\ref{Prop:HarmTupleConnection},
%	we can prove the theorem with an argument similar to that used to
%	prove~\cite[Proposition 3.9]{CMO20xx}.
%\end{proof}

%Using Lemma~\ref{Lem:ReInterBtf}, we have the following proposition.

\begin{prop}\label{Prop:ExtendedRephrase}
	We can {express} $J_{C,T}(w,z,x,y;q^m)$ as follows:
	\begin{multline*}
		J_{C,T}(w,z,x,y;q^m)
		= 
		\sum_{s=0}^{n-|T|}\,
		\sum_{t=0}^{|T|}\,
		\left(
			\sum_{X \in \binom{\overline{T}}{s}}\,
			\sum_{Y \in \binom{T}{t}} 
			(q^m)^{\ell_{T}(X,Y)}
		\right)\\
		(w-z)^{t}\, z^{|T|-t}\,
		(x-y)^{s}\, y^{n-|T|-s}.
	\end{multline*}
\end{prop}

\section{The connections}\label{Sec:Connection}

In this section, we establish two-way relationships between 
the Jacobi version of two weight polynomials, 
namely extended weight enumerators and higher weight enumerators.
Each of these relationships plays an important role in the development of the 
subsequent sections.

Let $\Mat(m\times n, \FF_{q})$ be the $\FF_{q}$-linear space 
of $m \times n$ matrices with entries in~$\FF_{q}$.
%Assume that $M \in \Mat(m\times n, \FF_{q})$.
%Then we denote by~$\mathrm{rw}(M,i)$ (resp.\ $\mathrm{col}(M,j)$) 
%the $i$-th row (resp. $j$-th column) of~$M$.
Then for an~$\FF_{q}$-linear code~$C$ of length~$n$,
$S_{m,n}(C)$ denotes the linear subspace of 
$\Mat(m\times n, \FF_{q})$ such that rows of each matrix are in~$C$. 
Then by~\cite[Proposition 5.27]{JP2013}, 
there is an isomorphism
%of $\FF_{q}$-vector spaces 
from $C{[q^m]}$ into~$S_{m,n}(C)$.
For a detailed discussion, we refer the reader to~\cite{JP2013, Simonis1993}.

\begin{lem}\label{Lem:wtTuwtTD}
	Let $\bm{u} \in C[\FF_{q^{m}}]$ and
	$M \in S_{m,n}(C)$ be the corresponding~$m\times n$
	matrix under a given isomorphism. 
	Let $D_{M} \subseteq C$ be the subcode generated by
	the rows of $M$. Then for any $T \subseteq [n]$,
	$\wt_{T}(\bm{u}) = \wt_{T}(D_{M})$.
\end{lem}

\begin{proof}
	Let $\alpha$ be a primitive $m$-th root of unity in $\FF_{q^m}$. 
	Then we can write an element of~$\FF_{q^m}$ 
	in a unique way on the basis $(1, \alpha, \alpha^2, \ldots, \alpha^{m-1})$
	with coefficients in $\FF_q$.
	Let $j$-th coordinate $u_{j}$ of $\bm{u}$ is zero, 
	then the $j$-th column of~$M$ consists of only zeros, 
	because the representation of~$u_{j}$ on the basis
	$(1, \alpha, \alpha^2, \ldots, \alpha^{m-1})$ is unique.
	On the other hand, if the $j$-th column of~$M$
	consists of all zeros, then $u_{j}$ is also zero.
	This implies for any $T\subseteq [n]$,
	$\supp_{T}(\bm{u}) = \Supp_{T}(D_{M})$.
	Hence $\wt_{T}(\bm{u}) = \wt_{T}(D_{M})$.
\end{proof}

\begin{thm}\label{Thm:ExtendedtoHigher}
	Let $C$ be an $[n,k]$ code over~$\FF_{q}$ and 
    let $T \subseteq [n]$. Then for any positive integer~$m$, 
    we have
	\[
		A_{i,j}^{q^m}(C;T)
		=
		\sum_{r=0}^{k}\,
		[m,r]_{q}\,
		A_{i,j}^{(r)}(C;T)\,.
	\]
	More specifically, we can write
	\[
		J_{C,T}
		(w,z,x,y;q^m)
		=
		\sum_{r=0}^{k}\,
		[m,r]_{q}\,
		J_{C,T}^{(r)}
		(w,z,x,y)\,.
	\]
\end{thm}

\begin{proof}
	Let $\bm{u} \in C{[q^m]}$ and $M \in S_{m,n}(C)$ be 
	the corresponding $m\times n$ matrix under a given 
	isomorphism $C[q^m] \to S_{m,n}(C)$. 
	Let $D_{M} \subseteq C$ be the subcode generated by	the rows of~$M$. 
	Then $\supp_{T}(\bm{u}) = \Supp_{T}(D_{M})$, 
	and hence by Lemma~\ref{Lem:wtTuwtTD}$, \wt_{T}(\bm{u}) = \wt_{T}(D_{M})$. 
    Therefore,
	\begin{align*}
		A_{i,j}^{q^m}&(C;T)\\
		& =
		\#
		\{
			\bm{u} \in C[\FF_{q^{m}}]
			\mid
			\wt_{[n]\backslash T}(\bm{u}) = i,
			\wt_{T}(\bm{u}) = j
		\}\\
		& =
		\sum_{M \in S_{m,n}(C)}
		\#
		\{
		D_{M} \subseteq C
		\mid
		\wt_{\overline{T}}(D_{M}) = i,
		\wt_{T}(D_{M}) = j
		\}\\
		& =
		\sum_{r = 0}^{k}\,
		\sum_{\substack{M \in S_{m,n}(C),\\ \mathrm{rk}({M}) = r}}\,
		\#
		\{
		D_{M} \in \D_{r}(C)
		\mid
		\wt_{\overline{T}}(D_{M}) = i,
		\wt_{T}(D_{M}) = j
		\}\\
		& =
		\sum_{r=0}^{k}\,
		[m,r]_{q}\,
		A_{i,j}^{(r)}(C;T)\,.\qedhere
	\end{align*}
\end{proof}

We can also express the higher Jacobi polynomials
in terms of the extended Jacobi polynomials.
The following result gives such correspondence.

\begin{thm}\label{Thm:HigherToExtended}
	Let $C$ be an $[n,k]$ code over~$\FF_{q}$ and
    let $T \subseteq [n]$. 
    Then for any integer~$r$ such that $0 \leq r \leq k$, 
    we have the following relation:
	\[
		J_{C,T}^{(r)}
		(w,z,x,y)
		=
		\dfrac{1}{[r]_{q}}
		\sum_{j=0}^{r}
		{\bbinom{r}{j}}_{q}
		(-1)^{r-j}
		q^{\binom{r-j}{2}}
		J_{C,T}
		(w,z,x,y;q^j).
	\]
\end{thm}

\begin{proof}
	By Theorem~\ref{Thm:HigherJacReinter}, Lemmas~\ref{Rem:TBritz} and \ref{Rem:BXrC}  
    and Proposition~\ref{Prop:ExtendedRephrase}, 
	\begin{align*}
		J_{C,T}^{(r)}&(w,z,x,y)\\
		& = 
		\sum_{s=0}^{n-|T|}\,
		\sum_{t=0}^{|T|}\,
		Q_{s,t}^{(r)}(C,T)\,
		(w-z)^{t}\,
		z^{|T|-t}\,
		(x-y)^{s}\,
		y^{n-|T|-s}\\ 
		& = 
		\sum_{s=0}^{n-|T|}\,
		\sum_{t=0}^{|T|}\,
		\sum_{X \in \binom{\overline{T}}{s}}\,
		\sum_{Y \in \binom{T}{t}}\,
		B_{X,Y}^{(r)}(C;T)\,
		(w-z)^{t}\,
		z^{|T|-t}\,\\
		&\hspace{150pt}
		(x-y)^{s}\,
		y^{n-|T|-s}
	\end{align*}
	\begin{align*}
		& = 
		\sum_{s=0}^{n-|T|}\,
		\sum_{t=0}^{|T|}\,
		\sum_{X \in \binom{\overline{T}}{s}}\,
		\sum_{Y \in \binom{T}{t}}\,
		{\bbinom{\ell_T(X,Y)}{r}}_{q}
		(w-z)^{t}\,
		z^{|T|-t}\,\\
		&\hspace{150pt}
		(x-y)^{s}\,
		y^{n-|T|-s}\\ 
		& = 
		\sum_{s=0}^{n-|T|}\,
		\sum_{t=0}^{|T|}\,
		\sum_{X \in \binom{\overline{T}}{s}}\,
		\sum_{Y \in \binom{T}{t}}\,
		\frac{[\ell_T(X,Y),r]_{q}}{[r]_{q}}
		(w-z)^{t}\,
		z^{|T|-t}\,\\
		&\hspace{150pt}
		(x-y)^{s}\,
		y^{n-|T|-s}\\
		& =
		\frac{1}{[r]_{q}}
		\sum_{s=0}^{n-|T|}\,
		\sum_{t=0}^{|T|}\,
		\sum_{X \in \binom{\overline{T}}{s}}\,
		\sum_{Y \in \binom{T}{t}}\, 
		\left(
		\sum_{j=0}^{r} 
		{\bbinom{r}{j}}_{q}
		(-1)^{r-j}
		q^{\binom{r-j}{2}}
		(q^{\ell_{T}(X,Y)})^j
		\right)\\
		& \hspace{150pt}
		(w-z)^{t}\,
		z^{|T|-t}\,
		(x-y)^{s}\,
		y^{n-|T|-s}\\
		& =
		\frac{1}{[r]_{q}}
		\sum_{j=0}^{r} 
		{\bbinom{r}{j}}_{q}
		(-1)^{r-j}
		q^{\binom{r-j}{2}}	
		\sum_{s=0}^{n-|T|}\,
		\sum_{t=0}^{|T|}\,
		\left(
		\sum_{X \in \binom{\overline{T}}{s}}\,
		\sum_{Y \in \binom{T}{t}}\,
		(q^{j})^{\ell_{T}(X,Y)}
		\right)\\
		& \hspace{150pt}
		(w-z)^{t}\,
		z^{|T|-t}\,
		(x-y)^{s}\,
		y^{n-|T|-s}\\
		& =
		\frac{1}{[r]_{q}}
		\sum_{j=0}^{r} 
		{\bbinom{r}{j}}_{q}
		(-1)^{r-j}
		q^{\binom{r-j}{2}}
		J_{C,T}(w,z,x,y;q^j).\qedhere		
	\end{align*}
\end{proof}

\begin{proof}[Proof of Theorem~\ref{Thm:HarmHigherMac}]
	By Theorems~\ref{Thm:HigherToExtended} and~\ref{Thm:JacExtMacWilliams.}, 
	we can write
	\begin{align*}
		J_{C^{\perp},T}^{(r)}
		& (w,z,x,y)\\
		& =
		\dfrac{1}{[r]_{q}}
		\sum_{j=0}^{r}
		{\bbinom{r}{j}}_{q}
		(-1)^{r-j}
		q^{\binom{r-j}{2}}
		J_{C^{\perp},T}
		(w,z,x,y;q^j)\\
		& = 
		\dfrac{1}{[r]_{q}}\,
		\sum_{j=0}^{r}\,
		{\bbinom{r}{j}}_{q}\,
		(-1)^{r-j}\,
		q^{\binom{r-j}{2}}\,
		\frac{1}{q^{jk}}\\
		& \hspace{50pt}
		J_{C,T}
		\left( 
		{w+(q^{j}-1)z, w-z}, {x+(q^{j}-1)y, x-y};
		q^j
		\right)\\
		& =
		\dfrac{1}{[r]_{q}}
		\sum_{j=0}^{r}
		{\bbinom{r}{j}}_{q}
		(-1)^{r-j}
		q^{\binom{r-j}{2}-jk}\\
		& \hspace{50pt}
		J_{C,T}
		\left( 
		{w+(q^{j}-1)z, w-z}, {x+(q^{j}-1)y, x-y};
		q^j
		\right)\\
	\end{align*}
	\begin{align*}
		& =
		\dfrac{1}{[r]_{q}}
		\sum_{j=0}^{r}
		{\bbinom{r}{j}}_{q}
		(-1)^{r-j}
		q^{\binom{r-j}{2}-jk}
		\sum_{\ell=0}^{j}\,
		[j,\ell]_{q}\\
		& \hspace{50pt}
		J_{C,T}^{(\ell)}
		\left(
		{w+(q^{j}-1)z}, 
		{w-z}, 
		{x+(q^{j}-1)y}, 
		{x-y}
		\right)\\
		& =
		\dfrac{1}{[r]_{q}}
		\sum_{j=0}^{r}\,
		\sum_{\ell=0}^{j}
		{\bbinom{r}{j}}_{q}
		(-1)^{r-j}
		q^{\binom{r-j}{2}-jk}\,
		[j,\ell]_{q}\\
		& \hspace{50pt}
		J_{C,T}^{(\ell)}
		\left(
		{w+(q^{j}-1)z}, 
		{w-z}, 
		{x+(q^{j}-1)y}, 
		{x-y}
		\right).
	\end{align*}
	The proof is completed by the following easily-proven identity:
	\[
	\dfrac{1}{[r]_{q}}
	{\bbinom{r}{j}}_{q}
	[j,\ell]_{q}
	=
	\frac{1}{q^{j(r-j)} [r-j]_{q} q^{\ell(j-\ell)} [j-\ell]_{q}}\,.\qedhere
	\]
	%It is not hard to prove this identity, so we omit it.	
\end{proof}

\begin{ex}\label{Ex:MacJac}
	Let $C$ be a $[6,3]$ code over $\FF_{2}$ with generator matrix
		\[
		\begin{pmatrix}
			1 & 1 & 0 & 0 & 0 & 0\\
			0 & 0 & 1 & 1 & 0 & 0\\
			0 & 0 & 0 & 0 & 1 & 1
		\end{pmatrix}.
		\]
		It is easy to check that~$C$ is self-dual.
		Again, there are seven linear subcodes with dimension~$1$.
		Three of them have weight~$2$, three of them have weight~$4$ 
		and one of them have weight~$6$.
		The support sets of weights $2$ and $4$
		are as follows:
		\begin{align*}
			\mathcal{S}_{1,2}(C)
			& =
			\{\{1,2\},\{3,4\},\{5,6\}\},\\
			\mathcal{S}_{1,4}(C)
			& =
			\{\{1,2,3,4\},\{1,2,5,6\},\{3,4,5,6\}\}\,.
		\end{align*}
		There are seven linear subcodes of~$C$ with dimension~$2$. 
		The generator matrices of the subcodes are listed below:
		
		$ $\\ [1mm]
		\begin{tabular}{ccc}
			\(\begin{pmatrix}
				1 & 1 & 0 & 0 & 0 & 0\\
				0 & 0 & 1 & 1 & 0 & 0
			\end{pmatrix}\)
			&
			\(\begin{pmatrix}
				1 & 1 & 0 & 0 & 0 & 0\\
				0 & 0 & 1 & 1 & 1 & 1
			\end{pmatrix}\)
			&
			\(\begin{pmatrix}
				1 & 1 & 0 & 0 & 1 & 1\\
				0 & 0 & 1 & 1 & 0 & 0
			\end{pmatrix}\)\\ \rule{0pt}{4ex}
			\(\begin{pmatrix}
				1 & 1 & 0 & 0 & 1 & 1\\
				0 & 0 & 1 & 1 & 1 & 1
			\end{pmatrix}\)
			&
			\(\begin{pmatrix}
				1 & 1 & 0 & 0 & 0 & 0\\
				0 & 0 & 0 & 0 & 1 & 1
			\end{pmatrix}\)
			&
			\(\begin{pmatrix}
				1 & 1 & 1 & 1 & 0 & 0\\
				0 & 0 & 0 & 0 & 1 & 1
			\end{pmatrix}\)\\ \rule{0pt}{4ex}
			\(\begin{pmatrix}
				0 & 0 & 1 & 1 & 0 & 0\\
				0 & 0 & 0 & 0 & 1 & 1
			\end{pmatrix}\) 
			& &
		\end{tabular}
	
		\noindent
		Three of them have weight~$4$; 
		the set of supports of these three is
		\[
		\mathcal{S}_{2,4}(C)
		=
		\{\{1,2,3,4\},\{1,2,5,6\},\{3,4,5,6\}\}\,.
		\] 
		Remaining four subcodes have weight~$6$. 
%		Note that $\mathcal{S}_{2,4}(C)$ is a $1$-design. 
		By direct calculation, we have
		\begin{align*}
			J_{C,\{i\}}^{(0)}
			(w,z,x,y)
			& =
			wx^{5},\\
			J_{C,\{i\}}^{(1)} 
			(w,z,x,y)
			& =
			zx^{4}y + 2wx^3y^2 + 2zx^2y^3 + wxy^4 + zy^5,\\
			J_{C,\{i\}}^{(2)}
			(w,z,x,y)
			& =
			wxy^{4} + 2zx^{2}y^{3}+6zy^5,
		\end{align*}
		for any coordinate position~$i$.
		Now by Theorem~\ref{Thm:HarmHigherMac}, we have
		\begin{multline*}
			J_{C^{\perp},\{i\}}^{(2)}
			(w,z,x,y)
			=
			\sum_{j = 0}^{2}
			\sum_{\ell = 0}^{j}
			(-1)^{3-j}
			\frac{2^{{2-j \choose 2}-j(2-j)-\ell(j-\ell)-j}}{[2-j]_{2} [j-\ell]_{2}}\\
			J_{C,\{i\}}^{(\ell)}
			\left( 
			{w+(2^{j}-1)z},\, 
			{w-z},\,
			{x+(2^{j}-1)y},\, 
			{x-y}
			\right).
		\end{multline*}
		This implies 
		\(
			J_{C^{\perp},\{i\}}^{(2)}
			(w,z,x,y)
			=
			wxy^{4} + 2zx^{2}y^{3}+6zy^5
		\)
		for any coordinate place~$i$.
\end{ex}

\section{Subcode designs and Jacobi polynomials}\label{Sec:Design}

Britz and Shiromoto~\cite{BrSh2008} gave a %celebrated 
generalization of the Assmus-Mattson Theorem~\cite{AsMa69}
by introducing designs from subcode supports of an $\FF_{q}$-linear code. 
Moreover, Bonnecaze, Mourrain and Sol\'e~\cite{BoMoSo1999} showed that 
if the set of codewords of a code form a $t$-design for 
every given weight, the Jacobi polynomial of the code
can be evaluated from the weight enumerator of the code
using polarization operator. 
In this section, we give a subcode design generalization 
of this result by presenting a formula 
that evaluate higher Jacobi polynomial from 
the higher weight enumerator by using polarization operator. 

A $t$-$(n, k, \lambda)$ design is a collection $\mathcal{B}$ of $k$-subsets (called blocks) 
of a set $E$ of $n$ elements,
such that any $t$-subset of $E$ is contained in exactly $\lambda$ blocks. 
A design is called \emph{simple} if the blocks are distinct; 
otherwise, the design is said to have \emph{repeated blocks}.

Bonnecaze, Mourrain and Sol\'e~\cite{BoMoSo1999} pointed out 
a remarkable characterization of codes supporting designs: 
the set of codewords of a code~$C$ form a $t$-design for every
fixed weight, if and only if the Jacobi polynomial 
$J_{C,T}$ for a $t$-set $T$ is independent of $T$.
In the following theorem we give the subcode design generalization
this result.
We omit the proof of the theorem since it follows from the definitions. 

\begin{thm}\label{Thm:JacToDesign}
	Let $C$ be an $[n,k]$ code over $\FF_{q}$
	and $i$, $r$ be positive integers such that
	$r \leq k$ and $i \ge d_{r}$. 
	Then for every given~$i$, $\mathcal{S}_{r,i}(C)$  
	forms  a $t$-design
	if and only if
	the $r$-th higher Jacobi polynomial 
	$J_{C,T}^{(r)}$ 
	attached to $t$-set $T \subseteq [n]$ 
	is independent of the choices of $T$.
\end{thm}

Let $C$ be a $[n,k]$ code over~$\FF_{q}$, and 
$r$ be a positive integer such that~$r \leq k$.
For any $X \subseteq [n]$, 
we denote by $W(X)$ the $n$-tuple 
$(w_{1},\ldots,w_{n}) \in \{0,1\}^{n}$
such that $w_{i} = 1$ if $i \in X$,
and $0$ otherwise.
We define
\begin{align*}
	\mathcal{C}_{r}(C)
	&:=
	\{W(X) \mid X \in \mathcal{S}_{r}(C)\}.
\end{align*}
\noindent
Then $\mathcal{C}_{r,i}^{0}(C)$ 
(resp. $\mathcal{C}_{r,i}^{1}(C)$) 
denotes the subsets of $\mathcal{C}_{r}(C)$, where
$i$-th entry of the elements of $\mathcal{C}_{r}(C)$ 
takes the value $0$ (resp.~$1$) punctured at~$i$.
The set obtained from $\mathcal{C}_{r}(C)$ 
by puncturing at coordinate
place~$i$ will be denoted by~$\mathcal{C}_{r,i}(C)$.
We say $\mathcal{C}_{r,i}(C)$ is
\textit{unique} if no matter $i$-coordinate gives the same $\mathcal{C}_{r,i}(C)$.
%A code is said to be $t-$\emph{homogeneous} if the codewords of every given
%weight hold a~$t-$design. In particular, we call a code \emph{homogeneous}
%instead of $t-$homogenous when $t =1$.

\begin{lem} \label{Lem:JacWeight}
	Let $C$ be an $[n,k]$ code over $\FF_{q}$
	and $r$ be a positive integers such that
	$r \leq k$.
	Then for every coordinate place~$i$, we have

	\begin{equation*}\label{Equ:Target_1}
		J_{C,\{i\}}^{(r)}
		=
		w 
		\sum_{\bm{u} \in \mathcal{C}_{r,i}^{0}(C)}
		x^{n-\wt(\bm{u})} y^{\wt(\bm{u})}
		+
		z 
		\sum_{\bm{u} \in \mathcal{C}_{r,i}^{1}(C)}
		x^{n-\wt(\bm{u})} 
		y^{\wt(\bm{u})}.
	\end{equation*}
\end{lem}

\begin{proof}
	It can be shown immediately from the definition.
\end{proof}

Now we describe 
the Aronhold polarization operator~$A$ presented in~\cite{BoMoSo1999}.
Let $f$ be a homogeneous polynomial in~$w$, $z$, $x$ and $y$.
We define the polarization operator $A$
for the homogeneous polynomial $f$ as
\[
	A . f 
	= 
	w\,
	\frac{\partial f}{\partial x}
	+
	z\,
	\frac{\partial f}{\partial y}
\]

Now we the following result.

\begin{thm} \label{Thm:HigherJac_1_Design}
	Let $C$ be an $[n,k]$ code over~$\FF_{q}$,
	and $r$, $j$ be a integer such that $r \leq k$
	and $j \geq d_{r}$.
	If $\mathcal{S}_{r,j}(C)$ forms a
	$1$-design for every given~$j$,  
	then for any coordinate place $i$,
	\[
		J_{C,\{i\}}^{(r)}
		=
		\frac{1}{n}\,
		AW_{C}^{(r)}.
	\]
\end{thm}

\begin{proof}
	Clearly,
	\begin{align*}
		\frac{\partial }{\partial x} 
		W_C^{(r)}
		& =
		\sum_{i=1}^{n}
		\left(
			\sum_{\bm{u} \in \mathcal{C}_{r,i}^{0}(C)}
			x^{n-\wt(\bm{u})} y^{\wt(\bm{u})}
		\right),\\
		\frac{\partial }{\partial y} 
		W_C^{(r)}
		& =
		\sum_{i=1}^{n}
		\left(
		\sum_{\bm{u} \in \mathcal{C}_{r,i}^{1}(C)}
		x^{n-\wt(\bm{u})} y^{\wt(\bm{u})}
		\right).
	\end{align*}
	\noindent
	Since $\mathcal{S}_{r,j}(C)$ forms a $1$-design
	for every given~$j$, 
	no matter $i$ gives the same $\mathcal{C}_{r,i}(C)$.
	Therefore,
	\begin{align*}
		\frac{1}{n}
		\frac{\partial }{\partial x} 
		W_C^{(r)}
		& =
		\sum_{\bm{u} \in \mathcal{C}_{r,i}^{0}(C)}
		x^{n-\wt(\bm{u})} y^{\wt(\bm{u})},\\
		\frac{1}{n}
		\frac{\partial }{\partial y} 
		W_C^{(r)}
		& =
		\sum_{\bm{u} \in \mathcal{C}_{r,i}^{1}(C)}
		x^{n-\wt(\bm{u})} y^{\wt(\bm{u})}.
	\end{align*}
	Then by Lemma~\ref{Lem:JacWeight}, we have
	\[
		J_{C,\{i\}}^{(r)}
		= 
		\frac{1}{n} 
		\left(
		w\, \frac{\partial }{\partial x}
		W_{C}^{(r)}
		+
		z\, \frac{\partial }{\partial y}
		W_{C}^{(r)}
		\right)
		=
		\frac{1}{n}\, AW_{C}^{(r)}.\qedhere
	\]
\end{proof}

\begin{ex}
	From Example~\ref{Ex:MacJac}, 
	we can see that
	$\mathcal{S}_{r,j}(C)$
	forms a $1$-design for $r = 1$, $j = 2, 4, 6$
	and $r = 2$, $j = 4, 6$.
	Moreover, the $r$-th higher weight enumerators 
	for $r = 1, 2$ are as follows:
	\begin{align*}
		W_{C}^{(1)}(x,y)
		& =
		3x^4y^2 + 3x^{2}y^{4} + y^{6},\\
		W_{C}^{(2)}(x,y)
		& =
		3x^{2} y^{4} + 4y^{6}.\\
	\end{align*}
	By Theorem~\ref{Thm:HigherJac_1_Design}, we have
	for any coordinate place~$i$:
	\begin{align*}
		J_{C,\{i\}}^{(1)}
		(w,z,x,y)
		& =
		zx^{4}y + 2wx^3y^2 + 2zx^2y^3 + wxy^4 + zy^5\\
		& =
		\frac{1}{6}
		AW_{C}^{(1)}(x,y),\\
		J_{C,\{i\}}^{(2)}
		(w,z,x,y)
		& =
		wxy^{4} + 2zx^{2}y^{3}+6zy^5\\
		& =
		\frac{1}{6}
		AW_{C}^{(2)}(x,y).
	\end{align*}
	
\end{ex}

The $t$-design generalization of Theorem~\ref{Thm:HigherJac_1_Design} 
is as follows:

\begin{thm}\label{Thm:HigherJac_t_design}
	Let $C$ be an $[n,k]$ code over~$\FF_{q}$,
	and $r, i$ be positive integers such that $r \leq k$
	and $i \geq d_{r}$.
	If $\mathcal{S}_{r,i}(C)$ forms a
	$t$-design for every given~$i$,  
	then for all $T \subseteq [n]$ with $|T|=t$,
	we have
	\[
		J_{C,T}^{(r)} 
		= 
		\frac{1}{n(n-1)\ldots (n-t+1)} 
		A^t\, W_C^{(r)}.
	\]
\end{thm}

\begin{proof}
	We prove by induction. 
	For $k<t$, we assume 
	\begin{equation*} \label{eq_prgenjac}
		J_{C,K}^{(r)}
		= 
		\frac{1}{n(n-1)\ldots (n-k+1)} 
		A^k\, W_C^{(r)},
	\end{equation*}
	where $K \subseteq [n]$ with $|K| = k$. 
	By deleting a coordinate $i$ on $[n]\backslash K$ 
	and move the coordinate $i$ to $K$,
	we have that
	\begin{align*}
		J_{C,K\cup \{i\}}^{(r)} 
		= & 
		\frac{1}{n-k} 
		\left(
			w\, \frac{\partial }{\partial x}
			J_{C,K}^{(r)}
			+
			z\, \frac{\partial }{\partial y}
			J_{C,K}^{(r)}
		\right) \\
		= & 
		\frac{1}{n-k} 
		\left( 
			A\, 
			\frac{1}{n(n-1) \cdots (n-k+1)} 
			A^k\, W_C^{(r)} 
		\right)\\
		= & \frac{1}{n(n-1) \cdots (n-k+1)(n-k)}\, 
		A^{k+1}\, W_C^{(r)}.\qedhere
	\end{align*}
\end{proof}

\section{Computation of higher Jacobi polynomials}\label{Sec:CompJac}

In this section, we show how to compute higher Jacobi polynomials 
using harmonic higher weight enumerators.
Introducing the concept of harmonic weight enumerator, 
Bachoc~\cite{Bachoc} gave a method that 
compute Jacobi polynomials
using harmonic weight enumerators through Hahn polynomials.
Recently, Britz, Chakraborty and Miezaki~\cite{BrChMi2026} 
presented the harmonic generalization of higher weight enumerators.
Therefore, it is very natural to investigate the computation of 
higher Jacobi polynomials using harmonic higher weight enumerators
through Hahn polynomials.

\subsection{Discrete Harmonic function}
Let
$\R 2^{[n]}$ and $\R \binom{[n]}{d}$
the real vector spaces spanned by the elements of  
$2^{[n]}$ and $\binom{[n]}{d}$,
respectively. 
An element of 
$\R\binom{[n]}{d}$
is denoted by
\begin{equation}\label{Equ:FunREd}
	f :=
	\sum_{Z \in \binom{[n]}{d}}
	f(Z) Z
\end{equation}
and is identified with the real-valued function on 
$\binom{[n]}{d}$
given by 
$Z \mapsto f(Z)$. 
Such an element 
$f \in \R \binom{[n]}{d}$
can be extended to an element 
$\widetilde{f}\in \R 2^{[n]}$
by setting, for all 
$X \in 2^{[n]}$,
\begin{equation}\label{Equ:TildeF}
	\widetilde{f}(X)
	:=
	\sum_{Z\in \binom{[n]}{d}, Z\subseteq X}
	f(Z).
\end{equation}
{Note that $\widetilde{f}(\emptyset) = f(\emptyset)$ when $d=0$,
	and that $\widetilde{f}(\emptyset) = 0$ otherwise}. 
If an element 
$g \in \R 2^{[n]}$
is equal to $\widetilde{f}$  
for some $f \in \R \binom{[n]}{d}$, 
then we say that $g$ has degree~$d$. 
{The differentiation operator $\gamma$ on $\mathbb{R}\binom{[n]}{d}$ 
	is defined by} linearity from the identity
\begin{equation}\label{Equ:Gamma}
	\gamma(Z) := 
	\sum_{Y\in \binom{[n]}{d-1}, Y\subseteq Z} 
	Y
\end{equation}
for all 
$Z \in \binom{[n]}{d}$
and for all $d=0,1, \ldots,n$.
Also, $\Harm_{d}(n)$ is the kernel of~$\gamma$:
\begin{equation}\label{Equ:Harm}
	\Harm_{d}(n) 
	:= 
	\ker
	\left(
	\gamma\big|_{\R \binom{[n]}{d}}
	\right).
\end{equation}

\begin{rem}\label{Rem:New}
	Let~$f \in \Harm_{d}(n)$. 
	Since $\sum_{Z \in \binom{[n]}{d}} f(Z) = 0$, 
	it is easy to check from~(\ref{Equ:Gamma}) that
	%$\sum_{X \subseteq E, |X| = t} \widetilde{f}(X) = 0$
	$\sum_{X \in \binom{[n]}{t}} \widetilde{f}(X) = 0$, where
	$1 \leq d \leq t \leq n$.
\end{rem}

The following theorem gives a remarkable characterization
of $t$-designs in terms of the harmonic spaces.

\begin{thm}[\cite{Delsarte}]\label{Thm:Delsert}
	Let $i,t$ be integers such that $0 \leq t \leq i \leq n$.
	A subset $\mathcal{B} \subseteq \binom{[n]}{i}$ is a $t$-design 
	if and only if
	\[
		\sum_{B\in \mathcal{B}}
		\left(
		\sum_{A \in \binom{[n]}{d},\, A \subseteq B} 
		f(A)) = 0
		\right)
	\]
	for all $f \in \Harm_{d}(n)$, $1 \leq d \leq t$.
\end{thm}

\begin{df}\label{DefHarmWeightBachoc}
	Let $C$ be an $[n,k]$ code over $\FF_{q}$ and let $f\in\Harm_{d}(n)$, where $d\neq 0$. 
	The \emph{harmonic $r$-th higher weight enumerator} of $C$ associated to $f$ is
	defined as follows:	
	\[
	W_{C,f}^{(r)}(x,y) 
	:=
	%		\sum_{{\bf u}\in C}
	%		\acute{f}({\bf u})
	%		x^{n-\wt({\bf u})}
	%		y^{\wt({\bf u})}
	%		=
	\sum^{n}_{i=0} 
	A_{i,f}^{(r)}(C) x^{n-i} y^{i},
	\]
	where 
	\[
	A_{i,f}^{(r)}(C)
	:= 
	\sum_{\substack{D \in \mathcal{D}_{r}(C),\\\wt(D) = i}} 
	\widetilde{f}(\supp(D))\,.
	\]
\end{df}

\begin{rem}
	Clearly,
	$A_{0,f}^{(r)}(C) = 0$ for all $0 \leq r \leq k$.
\end{rem}

Using Theorem~\ref{Thm:Delsert}, 
the following result characterizes 
the designs formed by subcode supports of a code
over~$\FF_{q}$ in term of harmonic higher weight enumerators. 

\begin{thm}[\cite{BrChMi2026}]\label{Thm:HarmDesign}
	Let $C$ be an $[n,k]$ code over~$\FF_{q}$.
	Let $r,i$ be the positive integers such that 
	$r \leq k$ and $i \geq d_{r}$.
	Then the set $\mathcal{S}_{r,i}(C)$ 
	forms a $t$-design 
	if and only if
	{$A_{i,f}^{(r)}(C) = 0$} for all
	$f \in \Harm_{d}(n)$, $1 \leq d \leq t$.
\end{thm}

\subsection{Hahn polynomials}

Hahn~\cite{Hahn1949} defined a remarkable polynomial,
now known as Hahn polynomial, that can 
form a family of orthogonal polynomials. 
%This polynomial is known as Hahn polynomial. 
Later, several properties of this polynomial were studied 
by Karlin and McGregor~\cite{Karlin}. 
In general, 
the Hahn polynomials can be defined
using the generalized hypergeometric series:
\[
	_3F_2(a_1,a_2,a_3;b_1,b_2;z) 
	:= 
	\sum_{i=0}^{\infty} 
	\dfrac{(a_1)_i(a_2)_i(a_3)_i}{(b_1)_i(b_2)_i}.
	\dfrac{z^i}{i!}
\]
where $ (a)_0 = 1 $ and $ (a)_i = a(a+1)(a+2) \dots (a+i-1) $ for $ i \geq 1 $. 
%The series terminates if one of the $ a_i $ is zero or a negative integer. 
For real $\alpha, \beta$ and for positive integer $ N $, 
Weber and Erd{\'e}lyi~\cite{WeEr1952} 
explicitly defined the Hahn polynomials 
$Q_m(x) \equiv Q_m(x; \alpha, \beta, N)$ 
as follows:
\begin{align*}
	Q_m(x) 
	& := 
	{_3F_2}(-m,\, -x,\, m+\alpha+\beta+1;\, \alpha+1,\, -N+1;\, 1)\\
	&\,\, =
	\sum_{i=0}^{m} 
	\dfrac{(-m)_i(-x)_i(m+\alpha+\beta+1)_i}{(\alpha+1)_i(-N+1)_i}
	\cdot
	\dfrac{1}{i!}
\end{align*}
for $ m = 0,1, \dots, N-1 $. 
From the above formula it can be seen 
that $Q_m(x)$ is a polynomial in the variable $x$ with degree exactly~$m$. 

\begin{rem}[\cite{Bartko}]
	Some special values of the Hahn polynomials are as follows:
	\begin{align*}
		Q_{m}(0; \alpha, \beta, N)
		& =
		1,
		\quad\quad
		Q_{0}(x; \alpha, \beta, N)
		=
		1\\[1mm]
		Q_{m}(N-1; \alpha, \beta, N)
		& =
		(-1)^m
		\binom{m+\beta}{m}{\bigg/}
		\binom{m + \alpha}{m}
	\end{align*}
\end{rem}

\begin{thm}[\cite{Bachoc}, Proposition 5.1]\label{ThProBachoc}
	Let $ T $ be a $t$-subset of $[n]$. 
	For all~$d$, $ 1 \leq d \leq t \leq n/2 $, 
	let $ H_{d,T} \in \R \binom{[n]}{t} $ be given by:
	\[
		H_{d,T}(X) 
		:= 
		Q_d^t(t - |X \cap T|)
	\]
	for all $t$-set $X$, where $ Q_d^t(x) \equiv Q_d(x; t-n-1, -t-1, t+1) $ are orthogonal Hahn polynomials. Then $ H_{d,T} \in \Harm_d(n) $.
\end{thm}

With the same hypothesis as above, 
an element of $\RR 2^{[n]}$,
$\widetilde{H}_{d,T}(X)$ for all subsets $X$ of $[n]$
only depends on $|X|$ and $|X\cap T|$.
We set 
$$\widetilde{H}_{d,T}(X) = h_{d,t}(|X|,|X\cap T|).$$
The explicit formula to compute $h_{d,t}(|X|,|X\cap T|)$
is given in~\cite[Proposition 5.2]{Bachoc} as follows:

\[
	h_{d,t}(\ell,i)
	=
	\frac{1}{\lambda_{d,t}}
	\sum_{I}
	v_{d,t}(\ell,i)\,
	Q_{d}^{t}
	(t-i_{1}-i_{2}),
\]
where
\begin{align*}
	I
	& =
	\begin{aligned}[t]
		\{i_1,i_2,i_3 \mid\,\,
		& 0 \leq i_1 \leq i,\\
		& 0 \leq i_2 \leq t-i,\\
		& 0 \leq i_3 \leq \ell -i,\\
		& i_1+i_2+i_3 \leq t,\\
		& i_1+i_3 \geq d \},
	\end{aligned}\\[1mm]
	\lambda_{d,t}
	& =
	\binom{n-2d}{t-d},\\[1mm]
	v_{d,t}(\ell,i)
	& =
	\binom{i}{i_{1}}
	\binom{t-i}{i_2}
	\binom{\ell-i}{i_3}
	\binom{n-\ell-t+i}{t-i_1-i_2-i_3}
	\binom{i_1+i_3}{d}.
\end{align*}

\subsection{Coefficients of $J_{C,T}^{(r)}$}

$ $\\[1mm]
Let $C$ be an $[n,k]$ code over $\FF_{q}$ and let $T$ be a $t$-subset of $[n]$.
Then the \emph{$r$-th higher Jacobi polynomial} of $C$ associated to $T$ 
can be redefine as follows:	
	\[
		J_{C,T}^{(r)}(w,z,x,y)
		=
		\sum_{\ell = 0}^{n}\,
		\sum_{i = 0}^{t}\,
		n_{\ell,i}^{(r)}(C,T)\,
		w^{t-i}\,
		z^{i}\,
		x^{n-t-\ell+i}\,
		y^{\ell-i},
	\]
	where
	\(
		n_{\ell,i}^{(r)}(C,T)
		:=
		\#
		\{
		D \in \D_{r}(C)
		\mid
		\wt(D) = \ell
		\text{ and }
		\wt_{T}(D) = i		
	\}.
	\)
On the other hand.
the harmonic higher weight enumerator~$W_{C,H_{d,T}}^{(r)}$
can be expressed as:
\begin{align*}
	W_{C,H_{d,T}}^{(r)}
	(x,y)
	& =
	\sum_{D \in \mathcal{D}_{r}(C)}
	\widetilde{H}_{d,T}(\Supp(D))\,
	x^{n-\wt(D)}
	y^{\wt(D)}\\
	& =
	\sum_{\ell=0}^{n}
	\left(
		\sum_{i = 0}^{t}
		h_{d,t}(\ell,i)\,
		n_{\ell,i}^{(r)}(C,T)
	\right)
	x^{n-\ell} y^{\ell}
\end{align*}	

Since $H_{d,T} \in \Harm_{d}(n)$, 
hence by Theorem~\ref{Thm:HarmDesign},
we have for all~$\ell$, the following $t$ linear equations
in $t+1$ unknowns $n_{\ell,i}^{(r)}(C,T)$, $0 \leq i \leq t$:
\begin{equation}\label{Equ:CoefJac1}
	\sum_{i = 0}^{t}
	h_{d,t}(\ell,i)\,
	n_{\ell,i}^{(r)}(C,T)
	=
	0,
	\quad
	\forall\,
	1 \leq d \leq t.
\end{equation}

For $d = 0$, we have the following equation for every~$\ell$:
\begin{equation}\label{Equ:CoefJac2}
	\sum_{i = 0}^{t}
	n_{\ell,i}^{(r)}(C,T)
	=
	A_{\ell}^{(r)}(C)
\end{equation}

Hence for all~$\ell$, $n_{\ell,i}^{(r)}(C,T)$
are the solutions of the system of linear equations~(\ref{Equ:CoefJac1})
and~(\ref{Equ:CoefJac2}).

\section*{Concluding remarks}
	The Assmus--Mattson theorem is one of the most
	important theorems in design and coding theory.
	Assmus--Mattson type theorems
	in lattices and vertex operator algebras are
	known as the Venkov~(cf.~\cite{Venkov}) and H\"ohn theorems (cf.~\cite{H1}), respectively; also see~\cite{Venkov2}. %\cite{{Venkov},{H1}}.
	For example, the $E_8$-lattice and moonshine vertex operator algebra
	$V^{\natural}$ provide
	spherical $7$-designs for all $(E_8)_{2m}$ and
	conformal $11$-designs for all $(V^{\natural})_{m}$, $m>0$.
	It is noteworthy that
	the $(E_8)_{2m}$ and $(V^{\natural})_{m+1}$
	are a spherical $8$-design and a conformal $12$-design, respectively, if and only if $\tau(m)=0$, 
	where
	$\tau(m)$ is known as Ramanujan's $\tau$-function satisfying
	\[
		q\prod_{m=1}^{\infty}(1-q^m)^{24}=\sum_{m=0}^{\infty}\tau(m)q^m;
	\]
	see~\cite{Miezaki}. Furthermore, Lehmer~\cite{Lehmer}
%	D.H.~Lehmer 
	conjectured 
%	in \cite{Lehmer} 
	that
	\(
		\tau(m)\neq 0
	\)
	for all $m$.  %\cite{{Miezaki},{Venkov},{Venkov2}}.
	Therefore, it is interesting to determine the lattice $L$
	(resp.~vertex operator algebra $V$) such that
	$L_m$ (resp.~$V_m$) are spherical (resp.~conformal) $t$-designs
	for all $m$ by the Venkov theorem (resp.~H\"ohn theorem) and
	$L_m'$ (resp.~$V_m'$) are spherical (resp.~conformal) $t'$-designs
	for some $m'$ with some $t'>t$.
	Our future studies will be inspired by these possibilities.
	For related results, see
	\cite{{BKST},{BM1},{Miezaki-Munemasa-Nakasora},{extremal design2 M-N}}.

%%%%%%%%%%%%%%%%%%%%%%%%%%%%%%%%%

\section*{Acknowledgements}
This work was supported by JSPS KAKENHI (22K03277) and SUST Research Centre (PS/2024/1/21).
%The authors would also like to thank the anonymous
%reviewers for their beneficial comments on an earlier version of the manuscript.

\section*{Data availability statement}
The data that support the findings of this study are available from
the corresponding author.


\begin{thebibliography}{99}

\bibitem{AsMa69} 
E.F.~Assmus, Jr.\ and H.F.~Mattson, Jr.,
New 5-designs,
{\sl J.~Comb. Theory}~{\bf 6} (1969), 122--151.

\bibitem{Bachoc}
C.~Bachoc,
On harmonic weight enumerators of binary codes,
%(English summary) Designs and codes--a memorial tribute to Ed Assmus. 
{\sl Des. Codes Cryptogr.} {\bf 18} (1999), 11--28.


%\bibitem{BachocNonBinary}
%C.~Bachoc,
%Harmonic weight enumerators of nonbinary codes and MacWilliams identities,
%{\sl Codes and Association Schemes (Piscataway, NJ, 1999)}, 1--23,
%{\sl DIMACS Ser. Discrete Math. Theoret. Comput.} {\bf 56},
%Amer. Math. Soc., Providence, RI, 2001. 

\bibitem{BKST}
E.~Bannai, M.~Koike, M.~Shinohara, and M.~Tagami, 
Spherical designs attached to extremal lattices and the modulo $p$ property of Fourier coefficients of extremal modular forms, 
{\sl Mosc.~Math.~J.} {\bf 6} (2006), 225--264.

\bibitem{BM1}
E.~Bannai and T.~Miezaki,
Toy models for D. H. Lehmer's conjecture.~
{\sl J.~Math.~Soc.~Japan} {\bf 62} (2010), no.~3, 687--705.


\bibitem{BO}
E.~Bannai and M.~Ozeki, 
Construction of Jacobi forms from certain combinatorial polynomials,
{\sl Proc.~Japan Acad.~Ser.~A Math.~Sci.} 
{\bf 72} (1996), no.~1, 12--15.


\bibitem{Bartko}
J.J.~Bartko,
Some summation theorems on the Hahn polynomials,
{\sl Amer. Math. Monthly} {\bf 70}(9) (1963), 978--981.


\bibitem{BoMoSo1999}
A.~Bonnecaze, B.~Mourrain and P.~Sol\'e,
Jacobi polynomials, type II codes, and designs, 
{\sl Des.~Codes Cryptogr.} {\bf 16} (1999), no.~3, 215--234.

%\bibitem{BMS1972}
%E.R.~Berlekamp, F.J.~MacWilliams and N.J.A.~Sloane, 
%Gleason's theorem on self-dual codes, 
%{\sl IEEE Trans. Inform. Theory} 
%{\bf IT-18} (1972), 409--414.

%\bibitem{Magma}
%W.~Bosma, J.~Cannon, and C.~Playoust, 
%The Magma algebra system. I. The user language, 
%{\em J. Symbolic Comput.}
%{\bf 24} (1997),  235--265.

%\bibitem{BrBrShSo07} 
%D.~Britz, T.~Britz, K.~Shiromoto and H.K.~S\o rensen,
%The higher weight enumerators of the doubly-even, self-dual $[48,24,12]$ code,
%{\sl IEEE Trans. Inform. Theory}~{\bf 53} (2007), 2567--2571.

%\bibitem{Britz2005}
%T.~Britz,
%Extensions of Critical theorem,
%{\sl Discrete Math.}
%{\bf 305} (2005), 55--73.

%\bibitem{Britz2007}
%T.~Britz,
%Higher support matroids,
%{\sl Discrete Math.}
%{\bf 307} (2007),
%2300--2308.

%\bibitem{BrCa22}
%T.~Britz and P.J.~Cameron,  
%Codes, 
%in 
%{\sl Handbook of the Tutte Polynomial and Related Topics (1st ed.)}
%(eds.\ J.A.~Ellis-Monaghan and I.~Moffatt), 328--344,
%Chapman and Hall/CRC, 2022.

%\bibitem{BrChIsMiTa2024}
%T.~Britz, H.S.~Chakraborty, R.~Ishikawa, T.~Miezaki and H.C.~Tang,
%Harmonic Tutte polynomials of matroids II,
%{\sl Des. Codes Cryptogr.} 
%{\bf 92} (2024), 1279--1297. 
%https://doi.org/10.1007/s10623-023-01343-0. 

%\bibitem{BJMS2012}
%T.~Britz, T.~Johnsen, D.~Mayhew and K.~Shiromoto,
%Wei-type duality theorems for matroids,
%{\sl Des. Codes Cryptogr.} {\bf 62} (2012), 331--341.

%\bibitem{BS2008}
%T.~Britz and K.~Shiromoto,
%A MacWilliams-type identity for matroids,
%{\sl Discrete Math.} {\bf 308} (2008), 4551--4559.

\bibitem{BrChMi2026}
T.~Britz, H.S.~Chakraborty and T.~Miezaki,
Harmonic higher and extended weight enumerators, 
{\sl J. Comb. Theory, Ser. A}
{\bf 217} (2026), Paper No. 106090.

\bibitem{BrSh2008}
T.~Britz and K.~Shiromoto,
Designs from subcode supports of linear codes,
{\sl Des. Codes Cryptogr.} {\bf 46} (2008), 175--189.

%\bibitem{BrRoSh2009} 
%T.~Britz, G.~Royle and K.~Shiromoto,
%Designs from matroids,
%{\sl SIAM J.~Disc. Math.}~{\bf 23} (2009), 1082--1099.

%\bibitem{BSW2015}
%T.~Britz, K.~Shiromoto and T.~Westerb\"ack,
%Demi-matroids from codes over finite Frobenius rings,
%(English summary) Designs and codes--a memorial tribute to Ed Assmus. 
%{\sl Des. Codes Cryptogr.} {\bf 75} (2015), 97--107. 

%\bibitem{ByCeIoJu2024} 
%E.~Byrne, M.~Ceria, S.~Ionica and R.~Jurrius,
%Weighted subspace designs from $q$-polymatroids,
%{\sl J.~Comb. Theory, Ser.~A} {\bf 201} (2024), 105799.


\bibitem{CHMO2025}
H.S.~Chakraborty, N.~Hamid, T.~Miezaki and M.~Oura,
{Jacobi polynomials, invariant rings, and generalized $t$-designs I},
{\sl Discrete Math.}  {\bf 348}(6) (2025), Paper No. 114447.

\bibitem{CIT2024}
H.S.~Chakraborty, R.~Ishikawa and Y.~Tanaka,
{Jacobi polynomials and design theory II},
{\sl Discrete Math.} {\bf 347}(3) (2024), Paper No. 113818.

\bibitem{CM2021}
H.S.~Chakraborty and T.~Miezaki,
Variants of Jacobi polynomials in coding theory,
{\sl Des. Codes Cryptogr.} {\bf 90} (2022), 2583-2597.


\bibitem{CMO2022}
H.S.~Chakraborty, T.~Miezaki and M.~Oura,
{Weight enumerators, intersection enumerators and Jacobi polynomials II},
{\sl Discrete Math.} {\bf 345}(12) (2022), 113098.


%\bibitem{CMO2023}
%H.S.~Chakraborty, T.~Miezaki and M.~Oura,
%Harmonic Tutte polynomials of matroids,
%{\sl Des. Codes Cryptogr.} {\bf 91} (2023), 2223--2226.

\bibitem{CMOT2023}
H.S.~Chakraborty, T.~Miezaki, M.~Oura and Y.~Tanaka,
{Jacobi polynomials and design theory I},
{\sl Discrete Math.} {\bf 346}(6) (2023), Paper No. 113339.


%\bibitem{Crapo}
%H.H.~Crapo, 
%The Tutte polynomial, 
%{\sl Aequationes Math.} {\bf 3} (1969), 211--229.

%\bibitem{CR1970}
%H.H.~Crapo and G.C.~Rota, 
%{\sl On the Foundations of Combinatorial Theory: Combinatorial Geometries}, 
%preliminary edition, The M.I.T. Press, Cambridge, MA and London, 1970.

%\bibitem{CS1999}
%J.H.~Conway and N.J.A.~Sloane, 
%{\sl Sphere Packings Lattices and Groups},
%third edition, Springer, New York, 1999.

\bibitem{Delsarte}
P.~Delsarte, 
Hahn polynomials, discrete harmonics, and $t$-designs, 
{\sl SIAM J.~Appl.\ Math.} {\bf 34} (1978), 157--166.


%\bibitem{Duke}
%W.~Duke, 
%On codes and Siegel modular forms, 
%{\sl Int.~Math.~Res.~Not.} 
%{\bf 5} (1993), 125--136.

\bibitem{EZ} 
M.~Eichler and D.~Zagier, 
{\sl The Theory of Jacobi Forms}, 
Progress in Mathematics, vol.~55, 
Birkhauser Boston, Inc., Boston, MA, 1985.


%\bibitem{Gleason} 
%A.M.~Gleason,
%Weight polynomials of self-dual codes and the MacWilliams identities, 
%in: {\sl Actes du Congr\`es International des Math\'ematiciens
%(Nice, 1970), Tome 3}, Gauthier-Villars, Paris, 1971, pp.~211--215.


%\bibitem{gould72}
%H.G.~Gould,
%{\sl Combinatorial Identities},
%Morgantown, WV, 1972.


%\bibitem{Greene1976}
%C.~Greene, 
%Weight enumeration and the geometry of linear codes, 
%{\sl Studia Appl. Math.}
%{\bf 55} (1976), 119--128. 


\bibitem{Hahn1949}
W.~Hahn, 
\"Uber Orthogonalpolynome, die $q$-differenzengleichungen  gen\"ugen, 
{\sl Math. Nach.} {\bf 2} (1949), 4--34.


\bibitem{H1}
G.~H\"ohn,
{Conformal designs based on vertex operator algebras},
{\sl Adv. Math.},
{\bf 217-5} (2008), 2301--2335.


\bibitem{HOO2020}
K. Honma, T.~Okabe and M.~Oura, 
Weight enumerators, intersection enumerators, and Jacobi
polynomials, 
{\sl Discrete Math.} 
{\bf 343}(6) (2020), Paper No. 111815.


\bibitem{HP2003}
W.~Huffman and V.~Pless,
{\sl Fundamentals of Error-Correcting Codes},
Cambridge University Press, Cambridge, 2003.



\bibitem{JP2013}
R.~Jurrius and R.~Pellikaan, 
Codes, Arrangements and Matroids,
{\sl Algebraic Geometry Modeling in Information Theory}, 
{\em Series on Coding Theory and Cryptology}
{\bf 8} (2013), 219-325.

\bibitem{Karlin}
S.~Karlin and J.~McGregor,
{The Hahn polynomials, formulas and an application},
{\sl Scripta Math.} {\bf 26} (1961), 33--46.

\bibitem{Klove1992}
T.~Kl{\o}ve, 
{Support weight distribution of linear codes}, 
{\sl Discrete Math.} 
{\bf 106/107} (1992), 311--316.

%\bibitem{Lam1999}
%T.Y.~Lam,
%{\sl Lectures on Modules and Rings},
%Springer-Verlag, New York, 1999.


\bibitem{Lehmer}
D.H.~Lehmer,
{The vanishing of Ramanujan's $\tau (n)$},
{\sl Duke Math.~J.}
{\bf 14} (1947), 429--433.


{\bibitem{vLiWi01} 
J.H.~van Lint and R.M.~Wilson, 
{\sl A Course in Combinatorics}, 2nd ed.,
Cambridge University Press, Cambridge, 2001.}

\bibitem{MacWilliams}
F.J.~MacWilliams, 
A theorem on the distribution of weights in a systematic code, 
{\sl Bell System Tech.~J.}
{\bf 42} (1963), 79--84. 

%\bibitem{MMS1972}
%F.J.~MacWilliams, C.L.~Mallows and N.J.A.~Sloane, 
%Generalizations of Gleason's theorem on weight enumerators of self-dual codes, 
%{\sl IEEE Trans. Inform. Theory} 
%{\bf IT-18} (1972), 794--805.

\bibitem{MS1977}
F.J.~MacWilliams and N.J.A.~Sloane, 
{\sl The Theory of Error-Correcting Codes}, 
North-Holland Editor, 1977.


\bibitem{Miezaki}
T.~Miezaki,
{Conformal designs and D.H. Lehmer's conjecture},
{\sl J.~Algebra} {\bf 374} (2013), 59--65.


\bibitem{Miezaki-Munemasa-Nakasora}
T.~Miezaki, A.~Munemasa and H.~Nakasora,
A note on Assmus--Mattson type theorems,
{\sl Des.~Codes Cryptogr.}, {\bf 89} (2021), 843--858. 

\bibitem{extremal design2 M-N}
T.~Miezaki and H.~Nakasora,
An upper bound of the value of $t$ of the support $t$-designs of extremal binary doubly even
self-dual codes,
\emph{Des. Codes Cryptogr.},
{\bf 79} (2016), 37--46.


%\bibitem{miezaki}
%T.~Miezaki, 
%On Eisenstein polynomials and zeta polynomials, 
%to appear in 
%{\sl Journal of Pure and Applied Algebra}.

%{
%\bibitem{miezakiweb}
%T.~Miezaki,\\
%Tsuyoshi Miezaki's website: 
%http://www.f.waseda.jp/miezaki/data.html
%}

%\bibitem{Molien}
%T.~Molien, 
%\"{U}ber die Invarianten der linearen Substitutionsgruppen, 
%{\sl Sitzungber.\ K\"{o}nig.\ Preuss.\ Akad.\ Wiss.} 
%{\bf 52} (1897), 1152--1156.

%\bibitem{Oura3}
%T.~Motomura, M.~Oura, 
%E-polynomials associated to $\ZZ_4$-codes, 
%{\sl Hokkaido Math.~J.} {\bf 47} (2018), no.~2, 339--350.

%\bibitem{Nozaki}
%H.~Nozaki, 
%A separation property of the zeros of Eisenstein series for $SL(2,\ZZ)$, 
%{\sl Bull.\ Lond.\ Math.\ Soc.} 
%{\bf 40} (2008), no.~1, 26--36. 


%\bibitem{Oura1}
%M.~Oura, 
%Eisenstein polynomials associated to binary codes, 
%{\sl Int.~J.~Number Theory} 
%{\bf 5} (2009), no.~4, 635--640. 

%\bibitem{Ouratalk}
%M.~Oura, 
%Talk at Oita National college of technology, (1st March, 2012). 

%\bibitem{Oura2}
%M.~Oura, 
%Eisenstein polynomials associated to binary codes (II), 
%{\sl Kochi J.~Math.} 
%{\bf 11} (2016), 35--41. 

%\bibitem{NRS}
%G.~Nebe, E.M.~Rains and N.J.A.~Sloane, 
%{\sl Self-Dual Codes and Invariant Theory}, 
%Algorithms and Computation in Mathematics, vol.~{\bf 17}, Springer-Verlag,
%Berlin, 2006.

%\bibitem{Oxley1992}
%J.G.~Oxley,
%{\sl Matroid Theory},
%Oxford University Press, New York, 1992.


\bibitem{Ozeki}
M.~Ozeki, 
On the notion of Jacobi polynomials for codes, 
{\sl Math.~Proc.~Cambridge Philos.~Soc.}
{\bf 121} (1) (1997) 15--30.

%\bibitem{Runge}
%B.~Runge, 
%Codes and Siegel modular forms, 
%{\sl Discrete Math.} {\bf 148} (1--3) (1996) 
%175--204.

%\bibitem{shiromoto96} 
%K.~Shiromoto,
%A new MacWilliams-type identity for linear codes,
%{\sl Hokkaido Math.~J.} {\bf 25} (1996), 651--656.


%\bibitem{Stanley1979}
%R.P.~Stanley,
%Invariants of finite groups and their applications to combinatorics,
%{\sl Bull. (New Series) Amer. Math. Soc.} {\bf 1}(3) (1979), 475--511.


\bibitem{Simonis1993}
J.~Simonis,
The effective length of subcodes,
{\sl AAECC} {\bf 5} (1993), 371--377.

%\bibitem{Tanabe2000}
%K.~Tanabe, 
%{An Assmus--Mattson theorem for $\ZZ_{4}$--codes}, 
%{\sl IEEE Trans. Inform. Theory}, 
%{\bf 46} (2000), 48--53.

\bibitem{Tanabe2001}
K.~Tanabe,
{A new proof of the Assmus--Mattson theorem for non-binary codes},
{\sl Des. Codes Cryptogr.} {\bf 22} (2001), 149--155.

%\bibitem{Tanabe2003}
%K.~Tanabe, 
%{A criterion for designs in $\ZZ_{4}$--codes on the
%symmetrized weight enumerator}, 
%{\sl Des. Codes Cryptogr.}
%{\bf 30} (2003), 169--185.

%\bibitem{Tutte1947}
%W.T.~Tutte, 
%A ring in graph theory, 
%{\sl Proc. Cambridge Philos. Soc.}
%{\bf 43} (1947), 26--40.

%\bibitem{Tutte1954}
%W.T.~Tutte, 
%A contribution to the theory of chromatic polynomial, 
%{\sl Canadian J.~Math.}
%{\bf 6} (1954), 80--91.

%\bibitem{Tutte1967}
%W.T.~Tutte, 
%On dichromatic polynomials, 
%{\sl J.~Comb. Theory}
%{\bf 2} (1967), 301--320.


\bibitem{Venkov}
{B.B.~Venkov, Even unimodular extremal lattices} (Russian),
{\sl Algebraic geometry and its applications. Trudy Mat. Inst. Steklov.}
{\bf 165} (1984), 43--48;
translation in
{\sl Proc.~Steklov Inst.~Math.}
{\bf 165} (1985) 47--52.

\bibitem{Venkov2}
B.B.~Venkov, R\'eseaux et designs sph\'eriques, (French) [Lattices and spherical designs] {\sl R\'eseaux euclidiens, designs sph\'eriques et formes modulaires}, 10--86, Monogr.~Enseign.~Math., 37, Enseignement Math., Geneva, 2001.


\bibitem{WeEr1952}
M.~Weber, and A.~Erd\'elyi,
On the Finite Difference Analogue of Rodrigues' Formula, 
{\sl Amer. Math. Monthly} {\bf 59}(3) (1952), 163--168.

\bibitem{Wei1991}
V.K.~Wei,
Generalized Hamming weights for linear codes,
{\sl IEEE Trans. Inf. Theory}
{\bf 37} (1991), 1412--1418.

%\bibitem{Yoshida1993}
%T.~Yoshida, 
%MacWilliams identities for linear codes with group action, 
%{\sl Kumamoto Math. J.} {\bf 6} (1993), 29--45.

%\bibitem{Mathematica}
%Wolfram Research, Inc., Mathematica, Version 11.2, Champaign, IL (2017).

\end{thebibliography}
\end{document}